\documentclass[11pt]{article}

\usepackage[dvipdfmx]{graphicx}
\usepackage{amssymb}
\usepackage{unit}
\usepackage{mathrsfs}
\usepackage{amsmath}
\usepackage{subfigure}
\usepackage{amsthm}
\usepackage{amscd}

\theoremstyle{definition}
\newtheorem{theorem}{Theorem}
\newtheorem{definition}[theorem]{Definition}

\newtheorem{remark}[theorem]{Remark}

\newcommand{\h}{{\bf h}}

\newcommand{\x}{{\bf x}}
\newcommand{\y}{{\bf y}}
\newcommand{\z}{{\bf z}}
\newcommand{\w}{{\bf w}}
\newcommand{\A}{{\bf A}}

\newcommand{\F}{{\bf F}}

\newcommand{\bS}{{\bf S}}

\newcommand{\V}{{\bf V}}
\newcommand{\bC}{\mathbb{C}}
\newcommand{\bN}{\mathbb{N}}
\newcommand{\bR}{\mathbb{R}}
\newcommand{\bZ}{\mathbb{Z}}
\newcommand{\cA}{\mathcal{A}}
\newcommand{\cB}{\mathcal{B}}
\newcommand{\sC}{\mathscr{C}}
\newcommand{\sD}{\mathscr{D}}
\newcommand{\sF}{\mathscr{F}}

\newcommand{\sL}{\mathscr{L}}

\title{%
Koopman resolvent: A Laplace-domain analysis\\of nonlinear autonomous dynamical systems
\thanks{The work was partially supported by JST, PRESTO Grant No.\,JPMJPR1926 (Y.S.), and 
the ARO-MURI grant W911NF-14-1-0359 (I.M.).}%
}%

\author{%
Yoshihiko Susuki\footnote{Corresponding author; E-mail to: susuki@eis.osakafu-u.ac.jp, susuki@ieee.org; Department of Electrical and Information Systems, Osaka Prefecture University, Japan, and JST, PRESTO, Japan.}
~~~
Alexandre Mauroy\footnote{Department of Mathematics and Namur Center for Complex Systems (naXys), University of Namur, Belgium.}
~~~
Igor Mezi\'c\footnote{Department of Mechanical Engineering, University of California, Santa Barbara, United States.}
}%

\date{}

\begin{document}
\maketitle

\begin{abstract}
The motivation of our research is to establish a Laplace-domain theory that provides principles and methodology to analyze and synthesize systems with nonlinear dynamics.  
A semigroup of composition operators defined for nonlinear autonomous dynamical systems---the Koopman 
semigroup and its associated Koopman generator---plays a central role in this study. 
We introduce the resolvent of the Koopman generator, which we call the Koopman resolvent, and provide its spectral characterization for three types of nonlinear dynamics: ergodic evolution on an attractor, convergence to a stable equilibrium point, and convergence to a (quasi-)stable limit cycle. 
This shows that the Koopman resolvent provides the Laplace-domain representation of such nonlinear autonomous dynamics.  
A computational aspect of the Laplace-domain representation is also discussed with emphasis on non-stationary Koopman modes. 
\end{abstract}

%\begin{keywords}
%Nonlinear dynamical systems; Koopman operator; Resolvent; Laplace transformation; Spectrum
%\end{keywords}

%\begin{AMS}
%37A30, 65P99
%\end{AMS}

%%%%%
\section{Introduction}
\label{sec:intro}

Over the last decade, the Koopman operator framework has attracted attention in theory and practice of nonlinear dynamical systems. It is however backed by a long tradition.
Bernard Koopman \cite{Koopman_PNAS17} proposed the use of linear operators on Hilbert spaces to study (nonlinear) Hamiltonian systems through the \emph{Koopman operator} and its spectrum.  
The Koopman operator is a composition operator based on a (possibly nonlinear) transformation \cite{Lasota:1994}.
Spectral analysis of the Koopman operator was traditionally conducted in ergodic theory of dynamical systems \cite{Arnold:1968,Peterson:1983,Lasota:1994} and has been more recently exploited to deal with various applications in science and engineering: see, e.g., \cite{Marko_CHAOS22,TheBook:2020}. 
It can capture the full information of the underlying nonlinear systems within a linear (but infinite-dimensional) setting and therefore mirrors the classical approach to linear systems. 
In particular, this approach allows to use spectral analysis for solving a nonlinear problem without lacking any information, in contrast to standard linearization techniques in state space.

In this paper, we focus on the use of spectral analysis of the Koopman operator to explore the Laplace-domain representation of nonlinear autonomous dynamical systems. 
The Laplace domain is a classical approach to analysis and synthesis of linear autonomous systems: see, e.g., \cite{Oppenheim:1997,Callier:1991,Hespanha:2009}. 
This approach provides a systematic method to characterize the systems through linear algebra and complex analysis techniques, and has been historically exploited in systems theory and signal processing. 
If a similar approach is established for nonlinear systems, it will provide principles and methodology to analyze (and hopefully synthesize) the systems in a systematic way. 

A real-world problem that motivates the nonlinear generalization of Laplace-domain theory is for example the engineering of complex power grids. 
Traditional Laplace-domain theory has been used for assessment and synthesis of power grid dynamic performance such as frequency stability and stabilization: see, e.g. \cite{Kundur_PSSC,Andersson_DCEPS}. 
This is successful for linear regime of power grid dynamics, that is, dynamics evolving in time in a small region close to a normal operating condition of a grid, which corresponds to an asymptotically stable equilibrium point of the associated dynamical model. 
However, this approach is not valid for dynamics evolving apart from the small region, which is referred to as the large-signal or transient stability problem and has been crucial to the occurrence of wide-spread disturbance \cite{Corsi_IEEEPESGM04,%Susuki_IEICETEA09,
Susuki_JNLS09}.
In this context, the systematic assessment and synthesis of power grid dynamic performance via the nonlinear generalization of Laplace-domain theory (naturally extended from the traditional linear one) would be desirable for practical use.

\begin{table}[t]
\centering
\caption{Laplace-domain representation of nonlinear autonomous dynamical  systems studied in this paper}
\label{tab:idea}
\begin{tabular}{ccc}\hline\hline\noalign{\vskip 1mm}
& Linear & Nonlinear
\\\hline\noalign{\vskip 1mm}
State-space representation & $\dot{\x}=\A\x$ & $\dot{\x}=\F(\x)$ \\
with scalar-valued output & $y={\bf c}^\top\x$ & $y=f(\x)$
\\\noalign{\vskip 1mm}\hline\noalign{\vskip 2mm}
Laplace transform $Y({s}; \x)$ of output $\{y(t)\}_{t\geq 0}$ & ${\bf c}^\top({s}{\bf I}-\A)^{-1}\x$ & $[{R} ({s}; {L})f](\x)$ \\
with domain in ${s}$, & ${s}\in\bC\setminus \sigma(\A)$ & ${s}\in\bC\setminus \sigma({L})$ \\
parameterized by initial state $\x$ & $\x\in\bR^n$ & $\x\in\cB\subseteq\bR^n$
\\\noalign{\vskip 2mm}\hline\hline
\end{tabular}
\end{table}

Our main idea on the Laplace-domain representation is summarized in Table\,\ref{tab:idea}, whose details will be explained throughout the paper.
When we focus on a nonlinear autonomous dynamical system with scalar-valued output, the Laplace transform of the output is represented by the so-called \emph{resolvent} (operator) $R({s}; {L})$ of the Koopman generator ${L}$ defined from the semigroup of Koopman operators, which we term as the \emph{Koopman resolvent}.  
This shows that the Laplace transform of the nonlinear output is clearly characterized with the system properties through the Koopman semigroup ${L}$. 
For systems with point and (quasi-)periodic attractors, formulae for spectral expansion of the Koopman semigroup ${L}$ are reported in \cite{Mezic_ARFM45,Mauroy_PD261,Mezic_JNS2019}. 
In this paper, we define the notion of Koopman resolvent $R({s}; {L})$, and building on these previous results, we provide expansion formulae of the Koopman resolvent $R({s}; {L})$ for the systems with: (i) dynamics on a (compact) ergodic (possibly aperiodic) attractor, and dynamics evolving to (ii) a stable equilibrium point and (iii) a limit cycle.
The expansion formulae do not rely on any linearization in state space and thus provide the Laplace transform of the nonlinear output, which is clearly an advantage against classical linearization in a neighborhood of the attractor. 
To the best of our knowledge, we propose the first study of the Koopman resolvent and its expansion formulae. 
The formulae could be used for structural analysis of the systems such as location of modes (poles), which mirror the classical approach to linear autonomous systems. 
In addition, we discuss a computational aspect of the Laplace-domain representation with emphasis on non-stationary parts of Koopman eigenvalues and Koopman modes \cite{Rowley_JFM641}.

The rest of this paper is organized as follows.  
In Section~\ref{sec:KO} we present a brief introduction to the Koopman operator for nonlinear autonomous systems.   
In Section~\ref{sec:resolvent} we introduce the Koopman resolvent as the key notion in this paper. 
In Section~\ref{sec:expansion} we state the main contribution of this paper and provide the expansion formulae of the Koopman resolvent for the systems with three types of nonlinear dynamics. 
A computational aspect of the obtained Laplace transforms for the nonlinear dynamics is discussed in Section~\ref{sec:comp}. 
Concluding remarks are given in Section~\ref{sec:outro}.

\emph{Notation}---All vectors are viewed as columns. 
We denote the sets of all natural numbers, of integers, and of non-negative integers by $\bN$, $\bZ$, and $\bN_0$, respectively. 
The sets of all real numbers and of all complex numbers are also denoted by $\bR$ and $\bC$.  
The spaces of square-summable functions, of analytic functions, and of functions with $k$ continuous derivatives are denoted by $\sL^2$, $\sC^\omega$, and $\sC^k$.  
For a vector $\x$, $\x^\top$ stands for its transpose. 
The imaginary unit is denoted by $\ii:=\sqrt{-1}$.

%%%%%
\section{The Koopman Operator and its Generator}
\label{sec:KO}

Throughout this paper, we consider a continuous-time, autonomous dynamical system described by the following nonlinear ordinary differential equation:
\begin{equation}
\dot{\x}=\F(\x) \qquad \forall\x\in\bR^n
\label{eqn:syst}
\end{equation}
where $\x$ is the state of the system, and $\F: \bR^n\to\bR^n$ is a nonlinear vector-valued function.  
We make the standing assumption that the solution
\[
\x(t)=\bS^t(\x_0) \qquad \forall t\geq 0
\]
associated with the initial condition $\x_0$ exists and is unique.
For instance, existence of the solution for all positive times can be obtained when the vector field is globally Lipschitz on a forward invariant set.  
The one-parameter family of maps $\bS^t: \bR^n\to\bR^n$, $t\geq 0$ is the \emph{flow} generated by \eqref{eqn:syst}. 
The system is usually described through its positive semi-orbit for the state, $\{\bS^t(\x_0)\}_{t\geq 0}$. 
Alternatively, one can focus on an output function $f: \bR^n\to\bC$, called \emph{observable}, and its positive semi-orbit $\{f(\bS^t(\x_0))\}_{t\geq 0}$. 
The evolution of all observables is represented by a one-parameter family of \emph{Koopman operators} for \eqref{eqn:syst}. 
\begin{definition}
\label{def:U}
Consider a (Banach) space $\sF$ of observables $f: \bR^n\to\bC$.  
The family of Koopman operators ${U}^t: \sF\to\sF$ associated with the family of maps $\bS^t:\bR^n\to\bR^n$, $t\geq 0$ is defined through the composition
\[
{U}^tf=f\circ\bS^t \qquad \forall f\in\sF. 
\]
\end{definition}

From the above argument, $\{\bS^t\}_{t\geq 0}$ is a semigroup, which implies that $\{{U}^t\}_{t\geq 0}$ is a semigroup of Koopman operators, called \emph{the Koopman semigroup}.  
It is important to note that the Koopman semigroup is \emph{linear} in the sense that each ${U}^t$ is linear. 
This allows to investigate the dynamics described by the nonlinear system \eqref{eqn:syst} through a linear semigroup. 
Note that due to existence and uniqueness of solutions of \eqref{eqn:syst}, $\{{U}^t\}_{t\in\bR}$ is equipped with the \emph{group} property. 
In the following we will keep the semigroup terminology to refer to semigroup theory of linear evolution equations \cite{Engel:2000,Kato:1995}. 
Also note that the Koopman semigroup is not always unitary for a wide class of systems as shown in this paper. 

If the semigroup $\{\bS^t\}_{t\geq 0}$ is continuously differentiable, then the Koopman semigroup $\{{U}^t\}_{t\geq 0}$ is strongly continuous in well-chosen spaces of observables (e.g. $\sC^0(\bR^2)$, $\sL^2(\bR^n)$), that is
\[
\lim_{t\downarrow 0}\|{U}^tf-f\|=0 \qquad \forall f \in \sF
\]
where $||\cdot||$ stands for the norm defined in $\sF$. 
This is stated in Proposition~16.2 (for $\sC^0$) and Exercise~16.5.4 (for $\sL^2$) in \cite{Batkai:2017} or Section~III-C %page~2554 
(for $\sL^2$) 
of \cite{Mauroy_IEEETAC65}. 
For a strongly continuous semigroup $\{{U}^t\}_{t\geq 0}$, the limit
\[
\lim_{t\downarrow 0}\frac{{U}^tf-f}{t} =: {L}f \qquad \forall f\in\sD 
\]
exists in the strong sense, where the domain $\sD$ of ${L}$ is a dense set in $\sF$. 
This defines the infinitesimal generator ${L}: \sD\to\sF$ of the Koopman semigroup, which is called the \emph{Koopman generator} and determines the time derivative of observables along the flow. 
The Koopman generator ${L}$ associated with \eqref{eqn:syst} is given by
\[
({L} f)(\x)=\F(\x)^\top{\bf \nabla}f(\x)
\]
where ${\bf \nabla}$ stands for the gradient operator in $\bR^n$.  
This explicitly connects the Koopman operator framework to the underlying system.

%%%%%
\section{Resolvent of the Koopman Generator}
\label{sec:resolvent}

Consider the resolvent set and spectrum of the Koopman generator ${L}$ for the nonlinear system \eqref{eqn:syst}. 
The generator ${L}$ is generally not bounded (and therefore not a continuous operator).  
The \emph{resolvent set} and \emph{spectrum} of a linear operator acting on a Banach space are defined (see, e.g., \cite{Curtain:1995,Engel:2000}) 
as follows.
\begin{definition}
The resolvent set $\rho({L})$ of the Koopman generator ${L}: \sD\to\sF$ is the set of complex values $\lambda$ such that the algebraic inverse of the operator $\lambda{I}-{L}$, denoted as $(\lambda{I}-{L})^{-1}$, exists and is a bounded linear operator on a dense domain of $\sF$, where ${I}$ is the identity operator.
\end{definition}
\begin{definition}
The spectrum $\sigma({L})$ of the Koopman generator ${L}: \sD\to\sF$ is defined to be
\[
\sigma({L}):=\bC\setminus\rho({L}).
\]
\end{definition}

For ${s}\in\rho({L})$, the bounded linear operator %inverse 
$({s}{I}-{L})^{-1}$ exists by definition. 
The \emph{resolvent} of a linear operator is stated as follows.
\begin{definition}
Suppose that ${\rho}({L})$ is not empty. 
The resolvent (operator) of the Koopman generator ${L}: \sD\to\sF$, which we call the Koopman resolvent, is defined as 
\begin{equation}
{R} ({s}; {L}):=({s}{I}-{L})^{-1} \qquad \forall{s}\in{\rho}({L}).
\end{equation}
\end{definition}

The spectrum $\sigma({L})$ is decomposed into the three disjoint sets: point, continuous, and residual spectra {\cite{Curtain:1995}}. 
Regarding the point spectrum {where $\lambda{I}-{L}$ is not injective}, a \emph{Koopman eigenvalue} of ${L}$ is a complex number $\lambda$ such that there exists a non-zero $\phi\in\sD\subset\sF$, called a \emph{Koopman eigenfunction}, such that ${L}\phi=\lambda\phi$. 
If the cardinality of Koopman eigenvalues is countably infinite, we use the integer subscript $j$ to represent a pair of Koopman eigenvalue and Koopman eigenfunction:
\[
{L}\phi_j=\lambda_j\phi_j, \qquad j=1,2,\ldots
\]
and equivalently, from the spectral mapping theorem \cite{Engel:2000}, 
\[
{U}^t\phi_j=\ee^{\lambda_jt} \phi_j \qquad \forall t\geq 0.
\]
{Note that the eigenfunctions are usually not mutually orthogonal.}  
A {complex value $\lambda\in\sigma({L})$} is in the continuous spectrum if {$\lambda{I}-{L}$ is injective, its range is dense in $\sF$, but $(\lambda{I}-{L})^{-1}$ is unbounded \cite{Curtain:1995}}. 
The residual spectrum does not appear in this paper and is not considered below. 

The key result used in this paper is from \cite{Engel:2000} and allows to connect the Koopman semigroup $\{{U}^t\}_{t\geq 0}$, the Koopman generator  ${L}$, and the associated Koopman resolvent ${R} ({s}; {L})$. 
In Proposition~5.5 of \cite{Engel:2000} for the strongly continuous $\{{U}^t\}_{t\geq 0}$ there exist constants $w\in\mathbb{R}$ and $M\geq 1$ such that
\[
\|U^t\| \leq M\ee^{wt} \qquad \forall t\geq 0
\]
where $\|{U}^t\|$ is the operator norm. 
Then, the following \emph{integral representation} of the Koopman resolvent ${R} ({s}; {L})$ holds (see Theorem 1.10 of \cite{Engel:2000}):
\begin{equation}
{R} ({s}; {L})f=\int^{\infty}_{0} \ee^{-{s} t}{U}^tf \, \dd t \qquad \forall f\in{\sF}%\sD
\label{eqn:Laplace}
\end{equation}
where the unilateral \emph{Laplace transform} on the right-hand side exists if $\mathrm{Re}[s]>w$.  
Note that for $s\notin {\rho}({L})$ a formula similar to \eqref{eqn:Laplace} allows to define Koopman eigenfunctions, which is called the generalized Laplace transform in \cite{Mezic_ARFM45}.

Next, consider the orbit $\{f(\bS^t(\x_0))\}_{t\geq 0}$ with initial condition $\x_0$ of the nonlinear system \eqref{eqn:syst}.  
The orbit is regarded as a scalar-valued signal of output or measurement from \eqref{eqn:syst} such as 
\[
y(t)=f(\bS^t(\x_0))=({U}^t f)(\x_0), \qquad \forall t\geq 0.
\]
If the space $\sF$ is a reproducing kernel Hilbert space, one can show that the unilateral Laplace transform of $\{y(t)\}_{t\geq 0}$ exists, which we denote by $Y({s}; \x_0)$. Indeed, evaluation functionals are bounded (i.e. there exists $C$ such that $|f(x)|\leq C \|f\|$ for all $f\in\sF$) and we have
\[
\int^\infty_0|y(t)\ee^{-st}|\dd t
= \int^\infty_0|({U}^tf)(\x_0)\ee^{-st}|\dd t
\leq C \int^\infty_0\|U^t f\|\, |\ee^{-st}|\dd t 
\leq C M \, \|f\| \int^\infty_0|\ee^{-(s-w)t}|\dd t
\]
so that the integral is absolutely convergent for $\mathrm{Re}[s]>w$.
Therefore, the unilateral Laplace transform $Y({s}; \x_0)$ is well-defined (see Section~II-3 of \cite{Widder:1946}).
In this case, the Laplace transform is directly related to \eqref{eqn:Laplace} as follows:
\begin{align}
Y({s}; \x_0)\, 
:= \int^\infty_0 \ee^{-st} y(t) \, \dd t 
= \int^\infty_{0} \ee^{-st} ({U}^t f)(\x_0) \, \dd t 
= [{R} ({s}; {L})f](\x_0) %. 
\label{eqn:KEY}
\end{align}
provided that $\mathrm{Re}[s]>w$. Note the improper integral in \eqref{eqn:Laplace} converges strongly, but not necessarily pointwise. 
Therefore, the last equality in \eqref{eqn:KEY} does not hold in general spaces, but holds in reproducing kernel Hilbert spaces where strong convergence implies pointwise convergence. 
Formula \eqref{eqn:KEY} shows that the Laplace transform $Y({s}; \x_0)$ of the output is determined by the action of the resolvent ${R} ({s}; {L})$ onto $f$, which is a generalization of the classical approach to linear systems as shown in Table\,\ref{tab:idea}.  

\begin{remark}
The above derivation provides a {dynamical interpretation of the Fourier spectra of nonlinear causal signals.  
For the Laplace transform} \eqref{eqn:KEY} { which exists as discussed above, it follows from Theorem~7.3 of \cite{Widder:1946} that if the output $\{y(t)\}_{t\geq 0}$ is of bounded variation in a neighborhood of $t\geq 0$ and continuous in $t$, then it can be} represented via inverse Laplace transform as
\begin{align}
\label{eq:Laplace_expansion}
y(t)=\int_{c-\ii \infty}^{c+\ii\infty} 
\ee^{st} \, [{R}({s}; {L})f](\x_0) \, \frac{\dd s}{2\pi\ii} \qquad \forall t> 0 %\geq 0
\end{align}
where {$c>w$, and the integral path is the vertical line from $c-\ii\infty$ to $c+\ii\infty$ in} the complex plane. 
This path is chosen so that {it lies on the right of the spectrum the Koopman generator ${L}$. 
Let us consider the condition $w<0$, so that
\[
\int^\infty_0|y(t)|\dd t
=\int^\infty_0|({U}^tf)(\x_0)|\dd t
\leq C\int^\infty_0\|U^t f\|  \dd t 
\leq C M \, \|f\| \int^\infty_0|\ee^{wt}|\dd t
\]
exists (where we used again the reproducing kernel Hilbert space property). 
This is well-known as a condition of absolute convergence of the Fourier integral of a causal signal ($y(t)=0$ for $t<0$).
Then, by setting $c=0$ and $s=\ii\omega$ in \eqref{eq:Laplace_expansion}, the output can be represented via the Fourier transform}
\begin{align}
y(t) = \int^\infty_{-\infty} \ee^{\ii \omega t}
\underbrace{\left[\int^\infty_0 \ee^{\ii\omega \tau} y(\tau) \, \dd\tau\right]}_{Y({\ii\omega}; \x_0)}
\frac{\dd\omega}{2\pi} 
= \int^\infty_{-\infty} \ee^{\ii \omega t}
[{R}({\ii\omega}; {L})f](\x_0) \, \frac{\dd\omega}{2\pi} \qquad \forall t{>0}. %\geq 0 .
\end{align}
It turns out that the Fourier spectrum of $y(t)$ corresponds to the value of $[{R}({\ii \omega}; {L})f](\x_0)$ along the imaginary axis $\ii\mathbb{R}$. 
%The Fourier transform can capture the spectral properties of ${L}$ for the unitary case because the spectrum has zero real-part: see Section~\ref{subsec:case-1} below.  
%However, the Fourier transformation does not necessarily work for the non-unitary case.
%One can consider the simple example $\dot{x}=-{\varepsilon} x$ ($x\in\bR$, ${\varepsilon}>0$), where the associated Koopman generator ${L}=(-{\varepsilon} x)\DD/\DD x$ is not unitary in the space $\sF=\sC^0(\bR)$. 
%Then, the Fourier transform fails to capture the Koopman eigenvalue $\lambda=-{\varepsilon}$ and is continuous in $\omega$. 
%In contrast, the Laplace transform is capable of capturing the spectral properties of ${L}$ even in the non-unitary case.
\end{remark}

%%%%%
\section{Expansion Formulae of the Koopman Resolvent}
\label{sec:expansion}

In this section, we provide formulae of spectral expansions of the Koopman resolvent for three types of nonlinear dynamics described by \eqref{eqn:syst}.

%%%
\subsection{On-Attractor Evolution}
\label{subsec:case-1}

First, we consider the case where the system \eqref{eqn:syst} possesses a compact attractor $\cA$ on which the dynamics is ergodic.  
In an ergodic attractor, almost all points are accessible in the sense that the initial conditions in this attractor sample the set. 
{This is characterized by the existence of an invariant measure on $\cA$.}  
Examples of ergodic attractors include stable limit cycles and quasi-periodic cycles (invariant tori). 
{Our analysis follows the traditional theory of ergodic dynamics in terms of the Koopman operator \cite{Arnold:1968}.  
One role of considering such ergodic invariant sets is that they are dynamically minimal in the indistinguishable sense \cite{Arnold:1968}.} 
By focusing on the asymptotic dynamics on $\cA$, the Koopman semigroup restricted to $\cA$, denoted by ${U}^t_\cA$ (with its generator ${L}_\cA$), is \emph{unitary} in the space $\sF=\sL^2(\cA)$, see {%page~316 of 
\cite{Koopman_PNAS17} or Theorem~9.3 of \cite{Arnold:1968}}. %, see \cite{Lasota:1994}. 
This kind of restriction of ${U}^t$ and ${L}$ is used for global stability analysis of nonlinear autonomous dynamical systems \cite{Alex_CDC13}.
{The $\sL^2$ space is defined with respect to the invariant measure.}  
Note that it is well-known that the semigroup ${U}^t_\cA$ is strongly continuous. This follows from the fact that the semigroup is bounded and strongly continuous on the dense subspace of continuous functions on $\cA$ (see Proposition I.5.3(c) of \cite{Engel:2000}). 
Then, the complete spectral expansion of ${U}^t_\cA$ is presented as follows:
\begin{equation}
{U}^t_\cA = \sum^\infty_{j=1} \ee^{{\ii\omega_j} t}{P}_{j} 
+ \int^{\infty}_{-\infty} 
\ee^{\ii {\omega} t} \, \dd{E}({\omega})\qquad \forall t\geq 0{.}
\label{eqn:on-attractor}
\end{equation}
{See Section~\ref{sec:AppA} for its derivation and Section~3.4 of} \cite{Mezic_ARFM45} for the case of Koopman group.\footnote{In \cite{Mezic_ARFM45}, the upper limit on the first term is $k$, that can be finite or infinite and the integral term has a typo in limits, where integral between $0$ and $1$, as in the discrete time case, is replaced by integral between $-\infty$ and $\infty$.} 
%where ${E}$ is a complex, continuous, operator-valued spectral measure on $\sL^2(\cA)$, which in the second term on the right-hand side represents the contribution of the continuous spectrum of ${L}_\cA$. 
%See \cite{Mezic_ARFM45} for the case of Koopman group.\footnote{In \cite{Mezic_ARFM45}, the upper limit on the first term is $k$, that can be finite or infinite and the integral term has a typo in limits, where integral between $0$ and $1$, as in the discrete time case, is replaced by integral between $-\infty$ and $\infty$.} 
%The infinite sum is convergent according to the spectral theorem for a semigroup of unitary operators \cite{Plesner:1969}. 
The first term of the right-hand side in \eqref{eqn:on-attractor} represents the contribution of the point spectrum of ${L}_\cA${. %, where 
The eigenvalue $\ii\omega_j$ is purely imaginary from the unitary ${U}^t_\cA$ and simple from the ergodicity assumption (see Theorem~9.21 of \cite{Arnold:1968}), and} ${P}_j$ is the {associated} orthogonal projection operator onto the eigenspace of ${L}_\cA$ associated with the Koopman eigenfunction $\phi_j$: for $f\in\sL^2(\cA)$,
\begin{equation}
{P}_jf=\phi_jV_j,\qquad V_j:=\bracket{f,\phi_j}_{\sL^2}
\label{eqn:projection}
\end{equation}
where $\bracket{\cdot,\cdot}_{\sL^2}$ stands for the inner product defined in $\sL^2(\cA)$, and 
$V_j$ is coined in \cite{Rowley_JFM641} as the \emph{Koopman mode}. 
{The second term of the right-hand side in \eqref{eqn:on-attractor} represents the contribution of the continuous spectrum of ${L}_\cA$, where ${E}$ is a continuous spectral measure on $\sL^2(\cA)$ (see Appendix in detail).
It is shown in \cite{Mezic_ARFM45} that in time-domain viewpoint, the first term (i.e., point spectrum of ${L}_\cA$) on the right-hand side in \eqref{eqn:on-attractor} governs the almost periodic part embedded in the output $y(t)=({U}^t_\cA f)(x)$, while the second term (i.e., continuous spectrum) does the aperiodic part in it.
}

Now, we derive the expansion formula of the Koopman resolvent ${R} ({s}; {L}_\cA)$ for the on-attractor evolution.
\begin{theorem}
Consider the system \eqref{eqn:syst} with a compact ergodic attractor $\cA$ and asymptotic dynamics on it.  
For the associated Koopman generator ${L}_\cA: \sD\to \sF=\sL^2(\cA)$, the following spectral expansion of the Koopman resolvent ${R} ({s}; {L}_\cA)$ holds:
\begin{equation}
{R} ({s}; {L}_\cA)=\sum^\infty_{j=1}\frac{{P}_j}{{s}-{\ii\omega_j}%\lambda_j
}
+ \int^\infty_{-\infty} 
\frac{\dd{E}({\omega}%\alpha
)}{{s}-\ii {\omega}%2\pi\alpha
}.
\label{eqn:expansion1}
\end{equation}
The Region-Of-Convergence (ROC) of ${R} ({s}; {L}_\cA)$ {is} %corresponds to 
$\{{s}\in\bC : \mathrm{Re}[{s}]>0\}$.
\end{theorem}
\begin{proof}
Applying \eqref{eqn:Laplace} to \eqref{eqn:on-attractor}, we obtain 
\begin{equation}
{R} ({s}; {L}_\cA) f 
= \int^\infty_{0} \ee^{-st}{U}^t_\cA f \, \dd t 
= \int^\infty_{0} \left[ \sum^\infty_{j=1} \ee^{({\ii\omega_j}%\lambda_j
-s) t} {P}_j f
+ \int^{\infty}_{-\infty}
\ee^{(\ii {\omega}%2\pi\alpha 
- s) t} \, \dd{E}({\omega}%\alpha
) f  \right]  \dd t 
\label{eqn:der2}
\end{equation}
for all $f \in \sL^2(\cA)$. 
{On the first term in the right-hand side of \eqref{eqn:der2}, we} %We 
have that, for some finite $N$,
\begin{align}
\left\|\sum^N_{j=1} \ee^{({\ii\omega_j}%\lambda_j
-s) t} {P}_j f \right\| 
&{= \ee^{-\mathrm{Re}[s]t}\left\|\sum^N_{j=1} \ee^{\ii\omega_j %\lambda_j
t} {P}_j f \right\|} \nonumber\\
&{\leq} 
\ee^{-\mathrm{Re}[s]t} \left\|\sum^\infty_{j=1} \ee^{{\ii\omega_j}%\lambda_j 
t} {P}_j f \right\| 
{\leq} \ee^{-\mathrm{Re}[s]t} \left\| U^t_\cA f \right\| 
= \ee^{-\mathrm{Re}[s]t}\|f\| \nonumber
\end{align}
where we used the facts that $P_j$ are orthogonal projections and that ${U}^t_\cA$ is unitary. 
Since $\ee^{-\mathrm{Re}[s]t}$ is integrable in $t\in[0,\infty)$ and the finite sum $ \sum^N_{j=1} \ee^{({\ii\omega_j}%\lambda_j
-s) t} {P}_j f$ is convergent as $N\to\infty$ if $\mathrm{Re}[s]>0${. 
Also, on the second term in the right-hand side of \eqref{eqn:der2}, we have that, for some finite $R$, 
\begin{equation*}
\left\|\int^R_{-R}\ee^{(\ii\omega-s)t}\dd E(\omega)f\right\| 
 = \ee^{-\mathrm{Re}[s]t} \left\|\int^R_{-R}\ee^{\ii\omega t}\dd E(\omega)f\right\|
%& \leq \ee^{-\mathrm{Re}[s]t} \left\|\int^{\infty}_{-\infty}\ee^{\ii t\omega}\dd E\sub{c}(\omega)f\right\| 
 \leq \ee^{-\mathrm{Re}[s]t}\|U^t_\mathcal{A} f\| = \ee^{-\mathrm{Re}[s]t} \|f\|
\end{equation*}
where we again used the facts that the spectral projections are orthogonal and that ${U}^t_\cA$ is unitary. 
Since $\ee^{-\mathrm{Re}[s]t}$ is integrable in $t\in[0,\infty)$ and the finite integral $\int^R_{-R}\ee^{(\ii\omega-s)t}\dd E(\omega)f$ is convergent as $R\to\infty$ if $\mathrm{Re}[s]>0$. 
Therefore}, we can apply the Lebesgue dominated convergence theorem (see, e.g., {Section~X of} \cite{Rudin:2005}) and obtain
\begin{equation*}
{R} ({s}; {L}_\cA) f 
= \sum^\infty_{j=1} \left[\int^\infty_{0} \ee^{({\ii\omega_j}%\lambda_j
-s) t} {P}_j f \, \dd t\right]
+ \int^{\infty}_{-\infty} 
\left[\int^\infty_{0} \ee^{(\ii {\omega} %2\pi\alpha 
- s) t} \, \dd t\right] \dd{E}({\omega}%\alpha
) f{,}    
\end{equation*}
so that the result follows. 
\end{proof}

{Since $\sL^2(\cA)$ is not a reproducing kernel Hilbert space, \eqref{eqn:KEY} does not hold for all functions $f\in\sL^2(\cA)$. 
One can circumvent this issue by considering the Sobolev space $\mathscr{H}^1(\cA)$ of functions whose derivatives belong to $\sL^2(\cA)$. 
This space is a reproducing kernel Hilbert space so that \eqref{eqn:KEY} holds. 
Following the same reasoning as in the case of $\sL^2(\cA)$ (based on Proposition I.5.3(c) of \cite{Engel:2000}), one can show that ${U}^t_\cA$ is strongly continuous in $\mathscr{H}^1(\cA)$, so that \eqref{eqn:Laplace} holds. 
Moreover, the resolvent coincides with the resolvent computed in $\sL^2(\cA)$ since strong convergence in $\mathscr{H}^1(\cA)$ implies strong convergence in $\sL^2(\cA)$. 
It follows that, for the on-attractor evolution, the unilateral Laplace transform $Y({s}; \x_0)$ of the output $\{y(t)=f(\x(t))\}_{t\geq 0}$ with initial condition $\x_0\in\cA$ and with $f\in \mathscr{H}^1(\cA) \subset \sL^2(\cA)$} is derived from \eqref{eqn:KEY}, \eqref{eqn:projection}, and \eqref{eqn:expansion1} as {follows:}
\begin{equation}
Y({s}; \x_0) = [{R} ({s}; {L}_\cA)f](\x_0) 
= \underbrace{\sum^\infty_{j=1}\frac{\phi_j(\x_0)V_j}{{s}-{\ii\omega_j}%\lambda_j
}}_{\displaystyle Y\sub{p}({s})}
+ \underbrace{\left[\int^\infty_{-\infty} 
\frac{\dd{E}({\omega}%\alpha
)}{{s}-\ii {\omega}%2\pi\alpha
}\right]f(\x_0)}_{\displaystyle Y\sub{c}({s})}. 
\label{eqn:expansion1-2}
\end{equation}
The function $Y\sub{p}({s})$ is associated with the almost periodic part of the {time} evolution {(output) $y(t)$}, 
while $Y\sub{c}({s})$ is associated with its aperiodic part. 
For mixing systems, there is no Koopman eigenfunction except a trivial constant function associated with eigenvalue 0, in which case the dynamics of the system are fully captured by $Y\sub{c}(s)$. 
By the above construction, $Y\sub{p}({s})$ is proved to be convergent{.} %and is meromorphic in the ROC (an open set) so that it can be written in the form $Y\sub{p}({s})=N\sub{p}({s})/D\sub{p}({s})$, where $N\sub{p}({s})$ and $D\sub{p}({s})$ are holomorphic in the ROC.
%The set of all {single} poles in $Y\sub{p}({s})$ (namely, all zeros of $D\sub{p}(s)$) is included in the point spectrum of the Koopman generator ${L}_\cA$.  
The residues of $Y\sub{p}({s})$ are the products of the values $\phi_j(\x_0)$ of the Koopman eigenfunctions at $\x_0$ with the Koopman modes $V_j$, which corresponds to the projection value $(P_j f)(\x_0)$ (see \eqref{eqn:projection}).

The characterization of $Y\sub{p}({s})$, which stems from the point spectrum of ${L}_\cA$, mirrors the classic characterization of linear autonomous systems via the eigenvalues of the system matrix $\A$ in Table\,\ref{tab:idea}.\footnote{It might be better to paraphrase that $Y\sub{p}({s})$ is analogue to the Laplace transform of the impulse response of a linear time-invariant system with single input and single output.}  
The transform \eqref{eqn:expansion1-2} suggests that the characterization of dynamics in Laplace domain is possible for the nonlinear system \eqref{eqn:syst}, even for a deterministic nonlinear system with aperiodic evolution. 

\begin{remark}
\label{rem:spec_expan}
In the case of isolated eigenvalues, the connection between the residues of $Y\sub{p}(s)$ and the eigenfunctions can also be obtained from Cauchy integral's formula. 
If $\lambda_j$ is isolated, the spectral projection is given by
\[
(P_jf)(\x_0)=\oint_{\gamma_j}[{R} ({s}; {L}_\cA)f](\x_0)\frac{\dd{s}}{2\pi\ii},
\]
where $\gamma_j$ is a small closed curve around $\lambda_j$, and is equal to the residue of $[{R} ({s}; {L}_\cA)f](\x_0)=Y({s}; \x_0)$ at $s=\lambda_j${, i.e. $(P_j f)(\x_0)=\mathrm{Res}_{\lambda_j} [{R} ({s}; {L}_\cA)f](\x_0)$.
Since the resolvent operator can be interpreted in terms of a Laplace transform (see \eqref{eqn:KEY}), this could yield a new interpretation of Generalized Laplace Analysis (GLA), which is typically used to compute the spectral projections $P_j f$ and relies on (Laplace) time averages (see, e.g.,  \cite{Marko_CHAOS22,Mauroy_PD261,Mezic_ARFM45}).}
Similarly, if the spectrum consists of isolated points, {it follows from the residue theorem that the integral} \eqref{eq:Laplace_expansion} can be expressed in terms of the residues of $[{R}({s}; {L}_\cA)f](\x_0)$, which directly leads to the spectral expansion {of the Koopman semigroup
\[
y(t)=({U}^tf)(\x_0)=\sum_{j} \ee^{\lambda_j t} \, \underset{\lambda_j}{\rm{Res}} [{R} ({s}; {L}_\cA)f](\x_0) .
\]
For this, the associated Laplace transform $Y\sub{p}(s)$ is meromorphic in the ROC (an open set) so that it can be written in the form $Y\sub{p}({s})=N\sub{p}({s})/D\sub{p}({s})$, where $N\sub{p}({s})$ and $D\sub{p}({s})$ are holomorphic in the ROC.
The set of all isolated, simple poles in $Y\sub{p}({s})$ (namely, all zeros of $D\sub{p}(s)$) is included in the point spectrum of the Koopman generator ${L}_\cA$.}

However, the residue theorem is typically derived for a finite or countable number of isolated poles. 
This is not always true in Koopman operator case (it can be true on discrete state spaces, where Koopman operator is a matrix: see \cite{mezic2019spectral,mezic_book}). 
Spectral expansions in this section have been obtained using {ergodic} %specific 
nature of the dynamics {as studied through} %in 
\cite{Mezic_JNS2019}, and thus \eqref{eq:Laplace_expansion} is related to spectral expansions in this more general setting. 
While the residue theorem is valid for functions that are holomorphic outside of the finite set of poles, the equivalence of \eqref{eq:Laplace_expansion} and spectral expansion is valid for functions in spaces defined in \cite{Mezic_JNS2019}: see Theorem~6.4 in Section~6. 
For instance, suppose we have a quasi-periodic torus as an attractor, with incommensurate basic frequencies $\omega_1$ and $\omega_2$. Then, $k_1\omega_1+k_2\omega_2$ can get arbitrarily close to zero for some $k_1,k_2 \in \mathbb{Z}$ and therefore $0$ is not an isolated singularity. 
However, as shown in {Section~9 of} \cite{Mezic_JNS2019}, the spectral expansion holds under KAM conditions for $\omega_1,\omega_2$. 
{This suggests that $Y\sub{p}(s)$ is generally not meromorphic for the nonlinear system \eqref{eqn:syst} unlike the LTI system with a linear observable.}
This %somehow 
provides a connection of residue theorem to dynamical systems theory. 
%It should be noted that the Laplace transform \eqref{eqn:expansion1-2} is derived without use of the Cauchy integral's formula. 
%In this sense, $Y\sub{p}({s})$ in \eqref{eqn:expansion1-2} can hold even for non-isolated eigenvalues with a slight modification, where a nilpotent part of the spectral expansion of ${L}$  appears as in \cite{Mezic_JNS2019}. 
\end{remark}

%%%
\subsection{Off-Attractor Evolution to Stable Equilibrium Point}
\label{subsec:case-2}

%%%
\subsubsection{Linear System}
\label{sssec:Linear}
We start the discussion with a linear stable system described by
\begin{equation}
\dot{\x}=\A\x \qquad \forall\x\in\mathbb{R}^n
\label{eqn:StableLinear}
\end{equation}
where $\A\in\mathbb{R}^{n\times n}$ is a constant matrix and assumed to have distinct $n$ eigenvalues $\lambda_1,\ldots,\lambda_n$ sorted according to
\[
\mathrm{Re}[\lambda_{j+1}]\leq
\mathrm{Re}[\lambda_j]\leq
\mathrm{Re}[\lambda_1]<0,\quad j=2,\ldots,n-1.
\]
For this system, the principal Koopman eigenfunctions $\phi_j(\x)$ associated with the Koopman eigenvalues $\lambda_j$ for $j=1,\ldots,n$ are explicitly constructed as (see \cite{Rowley_JFM641} and {Section~3.2 of} \cite{Mezic_JNS2019} for the case of generalized eigenfunctions with repeated eigenvalues)
\[
\phi_j(\x) = \w_j^\top\x
\]
where ${\bf v}_k$ are the right eigenvectors of $\A$ and $\w_j$ are the left eigenvectors of $\A$ (that is, $\w_j^\top\A=\lambda_j\w_j^\top$), normalized so that $\w_j^\top{\bf v}_k=\delta_{ik}$ (Kronecker's delta).  
Thus, since for any $\x$, we have
\[
\x = \sum^{n}_{j=1}\w^\top_j\x{\bf v}_j=\sum^{n}_{j=1}\phi_j(\x){\bf v}_j,
\]
the time evolution of a linear observable $y(t)=f(\x)={\bf c}^\top\x(t)$ (where ${\bf c}\in\mathbb{R}^{n}$, $t\geq 0$) along the trajectory of \eqref{eqn:StableLinear} starting from $\x=\x_0$ is represented through the Koopman operator ${U}^t$ by
\[
y(t) = ({U}^tf)(\x_0)
= \sum^{n}_{j=1}\ee^{\lambda_jt}\phi_j(\x_0) \, {\bf c}^\top{\bf v}_j \qquad \forall t\geq 0
\]
where ${\bf c}^\top{\bf v}_j$ ($j=1,\ldots,n$) coincide with the Koopman modes $V_j$.
{Note that the Koopman semigroup is naturally defined in the (Segal-Bargmann) space of analytic functions (see Section \ref{subsec:nonlinear_equil} for more details).}
Here, by applying the unilateral Laplace transform, it is possible to derive the Laplace transform $Y(s; \x_0)$ of the linear output $y(t)$ as
\[
Y(s; \x_0) = \sum^{n}_{j=1}\frac{\phi_j(\x_0)V_j}{s-\lambda_j}.
\]
This clearly shows that the poles of the Laplace transform are the principal Koopman eigenvalues of the stable linear system. 
Moreover, the residues correspond to the product of the Koopman modes $V_j$ with the values of the Koopman eigenfunctions at $\x_0$ . 
This observation will be generalized into the nonlinear system below.

%%%
\subsubsection{Nonlinear System}
\label{subsec:nonlinear_equil}

Next, we consider the nonlinear system \eqref{eqn:syst} possessing a stable equilibrium point $\x^\ast$ with a basin of attraction $\cB(\x^\ast)$.  
It is shown in \cite{Chiang_IEEETAC33} that $\cB(\x^\ast)$ is an open, invariant set which is diffeomorphic to $\bR^n$.  
Following \cite{Mauroy_PD261}, we assume that $\F$ is analytic ($\F \in \sC^\omega(\mathbb{R}^n)$) and that the Jacobian matrix computed at $\x^\ast$ has $n$ distinct (non-resonant) eigenvalues $\lambda_j$, characterized by strictly negative real parts and sorted according to
\begin{equation}
\mathrm{Re}[\lambda_{j+1}]\leq
\mathrm{Re}[\lambda_j]\leq
\mathrm{Re}[\lambda_1]<0,\quad j=2,\ldots,n-1.
\label{eqn:EigenValues}
\end{equation}
{The stability condition allows to obtain an ROC in the half-plane, which corresponds to a causal LTI system.
It should be also noted that the non-resonant condition is required by the Poincar\'e linearization theorem %[REF] 
to ensure that the eigenfunctions are analytic in some neighborhood of the equilibrium. 
This condition is necessary to obtain the spectral expansion below. }

Now, denote the Koopman semigroup restricted to $\cB(\x^\ast)$ by $\{{U}^t_\cB\}_{t\geq 0}$ (and generator ${L}_\cB$). 
{Note that the space of functions on which the semigroup is defined is given in the next paragraph.}
It is shown that the eigenvalues $\lambda_j$ {($j=1,\dots,n$)} are also the eigenvalues of {${L}_\cB$, which are called principal eigenvalues \cite{Ryan_Preprint:2014}.} 
The associated eigenfunctions $\phi_j$ {($j=1,\dots,n$)} are smooth { over $\cB(\x^\ast)$} and analytic in some neighborhood of the equilibrium {point $\x^\ast$} (see, e.g., \cite{Mauroy_PD261}). 
Moreover, using the Koopman eigenfunctions, we can define the conjugacy
\[
\h:\cB(\x^\ast) \to \mathbb{C}^n\,, \qquad \h(\x) = (\phi_1(\x),\dots,\phi_n(\x))^\top
\]
such that
\[
\h \circ \bS^t = \ee^{\mathbf{D} t} \circ \h
\]
where $\mathbf{D}:={\rm diag}(\lambda_1,\ldots,\lambda_n)$ is a diagonal matrix (see, e.g.,  \cite{Lan_PD242,Ryan_Preprint:2014}) {consisting of the principal Koopman eigenvalues}. 

{It is well-known that functions of the form $\{\phi_1(\x)\}^{k_1}\cdots\{\phi_n(\x)\}^{k_n}$, $k_1,\dots,k_n \in \mathbb{N}$, are also Koopman eigenfunctions, associated with the eigenvalues $k_1 \lambda_1+\cdots+k_n \lambda_n$ \cite{Mezic_ARFM45}. 
It is therefore natural to consider a space of functions that admit an expansion on these eigenfunctions, i.e. a space of functions that admit a polynomial decomposition in the new variables $\h$. 
Following Section~6 of \cite{Mezic_JNS2019}, an appropriate space is} the modulated Segal-Bargmann space $\mathcal{S}_{\h}$ of functions {that are analytic (entire) in $\h$,} i.e. of the form $f=g \circ \h$, where $g: \bC^n \to \bC$ is an analytic (entire) function. 
This space is endowed with the norm {and inner product:}
\begin{align}
\|f\|_{\mathcal{S}_{\h}} =\|g \circ \h\|_{\mathcal{S}_{\h}} & = \frac{1}{\pi^n} \int_{\bC^n} |g({\z}%\h
)|^2 \ee^{-|{\z}%\h
|^2}\dd{\z}%\h 
< \infty \,, \nonumber\\ 
\langle{f_1,f_2}\rangle 
& = \frac{1}{\pi^n}\int_{\bC^n} 
g_1(\z)\overline{g_2(\z)}\ee^{-|\z|^2}\dd\z \,, \qquad f_i=g_i\circ\mathbf{h} \,. \nonumber
\end{align}
Note that one could alternatively consider a (modulated) Hardy space to take into account functions that are analytic only in some neighborhood of the equilibrium \cite{Ryan_Preprint:2014}. 
Consider an analytic observable $f\in\mathcal{S}_{\h}$ which is {expanded in the basis of the Koopman eigenfunctions as} 
\[
f(\x) 
= \sum_{k_1,\ldots,k_n\in\bN}\{\phi_1(\x)\}^{k_1}\cdots\{\phi_n(\x)\}^{k_n}V_{k_1\cdots k_n}
\]
where the constant $V_{k_1\cdots k_n}$ is again the Koopman mode, so that 
\begin{equation}
({P}_{k_1\cdots k_n}f)(\x) 
= \{\phi_1(\x)\}^{k_1}\cdots\{\phi_n(\x)\}^{k_n}V_{k_1\cdots k_n}
\label{eqn:projection2}
\end{equation}
{defines} the (spectral) projection operator ${P}_{k_1\cdots k_n}$ similar to \eqref{eqn:projection}. 
{These projections are orthogonal in the modulated Segal-Bargmann space $\mathcal{S}_{\h}$ (but not in $\sL^2$).}
In fact, monomials in $\h$, i.e. $\{\phi_1(\x)\}^{k_1}\cdots\{\phi_n(\x)\}^{k_n}/\sqrt{k_1!\cdots k_n!}$, form an orthonormal basis in $\mathcal{S}_{\h}$, so that
\begin{equation}
\label{eq:norm_Foch}
\|f\|^2_{\mathcal{S}_{\h}} = \sum_{k_1,\ldots,k_n\in\bN} |V_{k_1\cdots k_n}|^2 k_1!\cdots k_n!\,.
\end{equation}  
The spectral expansion of ${U}^t_\cB$ for $f\in\mathcal{S}_{\h}$ is given as follows:
\begin{equation}
({U}^t_\cB f)(\x) 
= \sum_{k_1,\ldots,k_n\in\bN}\{\phi_1(\x)\}^{k_1}\cdots\{\phi_n(\x)\}^{k_n}V_{k_1\cdots k_n}
\ee^{(k_1\lambda_1+\cdots+k_n\lambda_n)t} \quad\forall t\geq 0.
\label{eqn:off-attractor-EP}
\end{equation}
{The condition $\mathrm{Re}[\lambda_j]<0$ for all $j$ implies that $\|{U}^t_\cB f\|^2_{\mathcal{S}_{\h}} \leq \|f\|^2_{\mathcal{S}_{\h}}$. Since the series \eqref{eqn:off-attractor-EP} is expanded in an orthonormal basis, it is strongly convergent for all $f\in \mathcal{S}_{\h}$ (see also Section~6 of \cite{Mezic_JNS2019}).} 
The series is also point-wise convergent due to the analyticity properties of the eigenfunctions and $f$ {(see \cite{Mauroy_PD261})}. 
{Moreover, the Koopman eigenfunctions form a complete orthonormal basis and the spectrum is totally disconnected. It follows that the Koopman generator satisfy the properties of a Riesz spectral operator, so that the Koopman semigroup $\{{U}^t_\cB\}_{t\geq 0}$ is strongly continuous (Theorem~2.3.5 in \cite{Curtain:1995}).}
{Note that the expansion \eqref{eqn:off-attractor-EP} shows that there is no continuous part of the spectrum.} 

The expansion formula of the Koopman resolvent ${R} ({s}; {L}_\cB)$ is given in the following theorem. 
\begin{theorem}
Consider the system \eqref{eqn:syst} with a stable equilibrium point $\x^\ast$ and dynamics converging to it. 
For the associated Koopman generator ${L}_\cB: \sD\to \mathcal{S}_{\h}$, the following spectral expansion of the Koopman resolvent ${R} ({s}; {L}_\cB)$ holds:
\begin{equation}
{R} ({s}; {L}_\cB)
= \sum_{k_1,\ldots,k_n\in\bN_0} 
\frac{{P}_{k_1\cdots k_n}}{{s}-(k_1\lambda_1+\cdots+k_n\lambda_n)}.
\label{eqn:expansion2}
\end{equation}
Also, the ROC of ${R} ({s}; {L}_\cB)$ {is} %corresponds to 
$\{{s}\in\bC : \mathrm{Re}[{s}]>0\}$.
\end{theorem}
\begin{proof}
Applying \eqref{eqn:Laplace} and \eqref{eqn:projection2} to \eqref{eqn:off-attractor-EP}, we obtain 
\[
{R} ({s}; {L}_\cB) f 
= \int^\infty_{0} \sum_{k_1,\ldots,k_n\in\bN_0} 
\ee^{(k_1\lambda_1+\cdots+k_n\lambda_n-s)t} {P}_{k_1\cdots k_n} f \, \dd t 
\]
for all $f \in \mathcal{S}_{\h}$. 
We have that, for some finite $N$,
\begin{equation*}
\begin{split}
\left\|\sum_{\substack{k_1,\ldots,k_n\in\bN_0 \\
k_1+\cdots+k_n \leq N}} \ee^{(k_1\lambda_1+\cdots+k_n\lambda_n-s)t} {P}_{k_1\cdots k_n} f \right\|^2_{\mathcal{S}_{\h}}
& {=} \sum_{\substack{k_1,\ldots,k_n\in\bN_0 \\
k_1+\cdots+k_n \leq N}}  \ee^{2\mathrm{Re}[k_1\lambda_1+\cdots+k_n\lambda_n-s]t} \\
& \makebox[12em]{} \times|V_{k_1\cdots k_n}|^2 k_1!\cdots k_n! \\
& \leq \ee^{-2\mathrm{Re}[s]} \sum_{k_1,\ldots,k_n\in\bN} |V_{k_1\cdots k_n}|^2 k_1!\cdots k_n! \\
&= \ee^{-2\mathrm{Re}[s]} \|f\|^2_{\mathcal{S}_{\h}}
\end{split}
\end{equation*}
where we used \eqref{eq:norm_Foch} and the fact that $\textrm{Re}[\lambda_j]<0$. 
Since $\ee^{-\mathrm{Re}[s]t}$ is integrable in $t\in[0,\infty)$ and the finite sum $\sum_{k_1+\cdots+k_n \leq N} \ee^{(k_1\lambda_1+\cdots+k_n\lambda_n-s)t} {P}_{k_1\cdots k_n} f$ is {strongly} convergent as $N\to\infty$ if $\mathrm{Re}[s]>0$, we can apply the Lebesgue dominated convergence theorem and obtain
\[
{R} ({s}; {L}_\cB) f 
=  \sum_{k_1,\ldots,k_n\in\bN_0} \int^\infty_{0}
\ee^{(k_1\lambda_1+\cdots+k_n\lambda_n-s)t} {P}_{k_1\cdots k_n} f \, \dd t{,}
\]
so that the result follows.
\end{proof}

{The modulated Segal-Bargmann space $\mathcal{S}_{\h}$ is a reproducing kernel Hilbert space (see e.g. Section~6 of \cite{Mezic_JNS2019}) so that \eqref{eqn:KEY} holds.}
For an analytic observable $f$ such that $f(\x^\ast)=0$, the unilateral Laplace transform $Y({s}; \x_0)$ of the output $\{y(t)=f(\x(t))\}_{t\geq 0}$ with initial condition $\x_0\in \cB(\x^\ast)\setminus\{\x^\ast\}$ is derived from \eqref{eqn:KEY}, \eqref{eqn:projection2}, and \eqref{eqn:expansion2} as
\begin{align}
Y({s}; \x_0) &= [{R} ({s}; {L}_\cB)f](\x_0) \nonumber\\
&
= \sum^{n}_{j=1}\frac{\phi_j(\x_0){V_{0\cdots (k_j=1)\cdots 0}}}{{s}-\lambda_j} 
+\sum_{\substack{k_1,\ldots,k_n\in\bN\\ k_1+\cdots+k_n>1}}
\frac{\{\phi_1(\x_0)\}^{k_1}\cdots\{\phi_n(\x_0)\}^{k_n}V_{k_1\cdots k_n}}
{{s}-(k_1\lambda_1+\cdots+k_n\lambda_n)}. 
\label{eqn:expansion2-2}
\end{align}
The ROC becomes $\{{s}\in\bC : \mathrm{Re}[{s}]>\mathrm{Re}[\lambda_1]\}$ from the order \eqref{eqn:EigenValues} of eigenvalues. 
The first term on the right-hand side of \eqref{eqn:expansion2-2} is related to the linearized system of \eqref{eqn:syst} around $\x^\ast$, but the coefficients $\phi_j(\x_0)V_{0\cdots (k_j=1)\cdots 0}$ depend on the initial condition $\x_0$ that is located in $\cB(\x^\ast)$, but possibly far from $\x^\ast$.
The second term does not appear in the linearized system and accounts for the non-stationary (wandering) dynamics due to the nonlinearity in \eqref{eqn:syst}. 
The Laplace transform $Y({s}; \x_0)$ is globally valid inside the basin $\cB(\x^\ast)$ of the nonlinear dynamical system \eqref{eqn:syst}, which cannot be achieved through linearization around $\x^\ast$ in the state space. 

\begin{remark}
For systems with an unstable equilibrium (e.g. hyperbolic saddle and source), an issue arises since the Laplace transform derived from the Koopman resolvent is not characterized by a proper region of convergence. 
For example, in the simple case $\dot{x}={\varepsilon} x$ (${\varepsilon}>0)$, the Koopman eigenvalues are given by $\lambda_n=n{\varepsilon}$ for any large positive integer $n$, so that the Laplace transform does not possess a region of convergence in the complex right half plane. 
{In other words, there is no spectral expansion of $y(t)=(U^t f)(\x_0)$ of the form \eqref{eq:Laplace_expansion} (see also Remark \ref{rem:spec_expan}) which is convergent for $t \geq 0$.}
This issue could be tackled by selecting appropriate function spaces where the spectral expansion is related to convergent series expansion of the flow (e.g., Laurent series).
\end{remark}

%%%
\subsection{Off-Attractor Evolution to Stable Limit Cycle}
\label{subsec:case-3}

Lastly, we consider the case of a system \eqref{eqn:syst} possessing a stable limit cycle $\cA$ with angular frequency ${\it\Omega}\,(>0)$ and basin of attraction $\cB(\cA)$.  
The topological property of $\cB(\cA)$ is the same as in $\cB(\x^\ast)$ (see \cite{Chiang_IEEETAC33}). 
It can be shown that the dynamics in $\cB(\cA)$ can be transformed to
\[
\left\{
\begin{aligned}
\dot{\y} & = \mathbf{A}(\theta) \, \y \\
\dot{\theta} & = \it{\Omega}
\end{aligned}
\right.
\]
with $\y \in \mathbb{R}^{n-1}$, $\theta \in \mathbb{T}$, and ${A}(\theta) \in \mathbb{R}^{(n-1)\times(n-1)}$ is $2\pi$-periodic \cite{Lan_PD242}. 
Let the eigenvalues ${\mu}_1,\ldots,{\mu}_{n-1}$ of the Floquet stability matrix $\mathbf{B}$ be distinct, strictly positive, and sorted so that 
\begin{equation}
|{\mu}_{j+1}|\leq
|{\mu}_{j}|\leq
|{\mu}_1|<1,\quad j=2,\ldots,n-2.
\label{eqn:Multipliers}
\end{equation}
These values are the characteristic multipliers of the limit cycle, and any $\nu_j$ {satisfying} ${\mu}_j=\ee^{2\pi\nu_j/{\it\Omega}}$ is called a characteristic exponent. 

It is shown in {Theorem~8.1 of} \cite{Mezic_JNS2019} that the characteristic exponents are also the eigenvalues of ${U}^t_\cB$. 
The corresponding eigenfunctions are the components of\\
$\z(\y,\theta)=(z_1(\y,\theta), \cdots, z_{n-1}(\y,\theta))^\top$, with 
\[
\z(\y,\theta) = \V^{-1} \bf P^{-1}(\theta) \y
\]
and where $\bf P(\theta)$ is the Floquet matrix and $\V$ is the diagonalizing matrix for $\mathbf{B}$. Note that 
\[
\phi_{k_1\cdots k_{n-1},{m}}(\x) = z_1^{k_1}(\y,\theta) \cdots z_{n-1}^{k_{n-1}}(\y,\theta) \ee^{\ii m\theta}
\]
is a Koopman eigenfunction associated with the Koopman eigenvalue $k_1 \nu_1+\cdots+k_{n-1} \nu_{n-1} + \ii m \Omega$. 

Following {Remark~8.1 of} \cite{Mezic_JNS2019} and \cite{Ryan_Preprint:2014}, we consider the space {$\mathcal{S}_{\z,\theta}$} of functions that are analytic (entire) in $\z\in \mathbb{C}^{n-1}$ and $\sL^2$ in $\theta \in \mathbb{T}$, endowed with the norm
\[
\|f\|_{{\mathcal{S}_{\z,\theta}}} = \frac{1}{\pi^{n-1}} \int_{\mathbb{T}} \int_{\mathbb{C}^{n-1}} |f(\z,\theta)|^2 \ee^{|\z|^2}\dd\z \,\dd\theta \qquad \forall f\in\mathcal{S}_{{\z,\theta}}.
\]
The space {$\mathcal{S}_{\z,\theta}$} is called in {Section~7 of} \cite{Mezic_JNS2019} the averaging kernel Hilbert space. 
The norm can be rewritten as
\[
\|f\|_{{\mathcal{S}_{\z,\theta}}}^2 = \sum_{\substack{k_1,\ldots,k_{n-1}\in\bN_0\\ {m}\in\bZ}}  |V_{k_1\cdots k_{n-1},{m}}|^2 k_1!\cdots k_{n-1}!
\]
for $f\in {\mathcal{S}_{\z,\theta}}$ that admits its expansion as follows:
\[
f(\x) = \sum_{\substack{k_1,\ldots,k_{n-1}\in\bN_0\\ {m} \in\bZ}}  \phi_{k_1\cdots k_{n-1},{m}}(\x) V_{k_1\cdots k_{n-1},{m}}. 
\]
The constant $V_{k_1\cdots k_{n-1},{m}}$ is again the Koopman mode and related to the {orthogonal} projection operator ${P}_{k_1\cdots k_{n-1},{m}}$, so that
\begin{equation}
({P}_{k_1\cdots k_{n-1},{m}}f)(\x)
=\phi_{k_1\cdots k_{n-1},{m}}(\x)V_{k_1\cdots k_{n-1},{m}}
\label{eqn:projection3}
\end{equation}
for $f\in {\mathcal{S}_{\z,\theta}}$. 
The spectral expansion of ${U}^t_\cB$ for $f\in {\mathcal{S}_{\z,\theta}}$ is given in {equation~(172) of} \cite{Mezic_JNS2019} by
\begin{align}
({U}^t_\cB f)(\x) =& \sum_{\substack{k_1,\ldots,k_{n-1}\in\bN_0\\ {m}\in\bZ}}
\phi_{k_1\cdots k_{n-1},{m}}(\x) V_{k_1\cdots k_{n-1},{m}}\times \nonumber\\
& \makebox[10em]{}\times 
\ee^{[\{(k_1\nu_1+\cdots+k_{n-1}\nu_{n-1})+\ii{m}{\it\Omega}\}t]} 
\qquad \forall t\geq 0.
\label{eqn:off-attractor-LC}
\end{align}
{The condition $\mathrm{Re}[\nu_j]<0$ for all $j$ implies that $\|{U}^t_\cB f\|^2_{\mathcal{S}_{\z,\theta}} \leq \|f\|^2_{\mathcal{S}_{\z,\theta}}$. Since the series \eqref{eqn:off-attractor-LC} is expanded in an orthonormal basis, it is strongly convergent for all $f\in \mathcal{S}_{\z,\theta}$.}
{Moreover, the Koopman eigenfunctions form a complete orthonormal basis and the spectrum is totally disconnected. It follows that the Koopman generator satisfy the properties of a Riesz spectral operator, so that the Koopman semigroup $\{{U}^t_\cB\}_{t\geq 0}$ is strongly continuous (Theorem~2.3.5 in \cite{Curtain:1995}).}
The expansion formula of the Koopman resolvent ${R}({s}; {L}_\cB)$ is given in the following theorem. 
\begin{theorem}
\label{thm:Stable-LC}
Consider the system \eqref{eqn:syst} with a stable limit cycle $\cA$ and dynamics converging to it.  
For the associated Koopman generator ${L}_\cB: \sD\to {\mathcal{S}_{\z,\theta}}$, the following spectral expansion of the Koopman resolvent ${R} ({s}; {L}_\cB)$ holds:
\begin{align}
{R} ({s}; {L}_\cB) &= \sum_{\substack{k_1,\ldots,k_{n-1}\in\bN_0\\ {m}\in\bZ}} {P}_{k_1\cdots k_{n-1},{m}}
\frac{1}{{s}-\{(k_1\nu_1+\cdots+k_{n-1}\nu_{n-1})+\ii{m}{\it\Omega}\}}. 
\label{eqn:expansion3}
\end{align}
Also, the ROC of ${R} ({s}; {L}_\cB)$ {is} %corresponds to 
$\{{s}\in\bC : \mathrm{Re}[{s}]>0\}$. 
\end{theorem}
\begin{proof}
Applying \eqref{eqn:Laplace} and \eqref{eqn:projection3} to \eqref{eqn:off-attractor-LC}, we obtain
\begin{equation*}
{R} ({s}; {L}_\cB) f 
= \int^\infty_{0} \sum_{\substack{k_1,\ldots,k_{n-1}\in\bN_0\\ {m}\in\bZ}}
\ee^{[\{(k_1\nu_1+\cdots+k_{n-1}\nu_{n-1})+\ii{m}{\it\Omega}-s\}t]} {P}_{k_1\cdots k_{n-1},{m}}f \, \dd t 
\end{equation*}
for all $f \in {\mathcal{S}_{\z,\theta}}$. We have that, for some finite $N$,
\begin{equation*}
\begin{split}
& \left\|\sum_{\substack{k_1,\ldots,k_{n-1}\in\bN_0\\ {m}\in\bZ}} \ee^{[\{(k_1\nu_1+\cdots+k_{n-1}\nu_{n-1})+\ii{m}{\it\Omega}-s\}t]} {P}_{k_1\cdots k_{n-1},{m}}f \right\|^2_{{\mathcal{S}_{\z,\theta}}} \\
& \quad \leq \sum_{\substack{k_1+\cdots+k_{n-1}\leq N\\ |m|\leq M}}
\ee^{2\mathrm{Re}[k_1\nu_1+\cdots+k_{n-1}\nu_{n-1}-s]t} |V_{k_1\cdots k_{n-1},m}|^2 k_1!\cdots k_{n-1}! \\
& \quad \leq \ee^{-2\mathrm{Re}[s]} \sum_{\substack{k_1,\ldots,k_{n-1}\in\bN_0\\ {m}\in\bZ}} |V_{k_1\cdots k_{n-1},m}|^2 k_1!\cdots k_{n-1}! = \ee^{-2\mathrm{Re}[s]} \|f\|^2_{{\mathcal{S}_{\z,\theta}}} \,.
\end{split}
\end{equation*}
Since $\ee^{-\mathrm{Re}[s]t}$ is integrable in $t\in[0,\infty)$ and the finite sum is convergent if $\mathrm{Re}[s]>0$, we can apply the Lebesgue dominated convergence theorem and obtain
\begin{equation*}
{R} ({s}; {L}_\cB) f 
= \sum_{\substack{k_1,\ldots,k_{n-1}\in\bN_0\\ {m}\in\bZ}} \int^\infty_{0}
\ee^{[\{(k_1\nu_1+\cdots+k_{n-1}\nu_{n-1})+\ii{m}{\it\Omega}-s\}t]} {P}_{k_1\cdots k_{n-1},{m}}f \, \dd t 
\end{equation*}
so that the result follows. 
\end{proof}

For illustration purpose, let us consider the planar case with $n=2$.  
The single characteristic multiplier ${\mu} \in \mathbb{R}$ and the associated exponent $\nu \in \mathbb{R}$ quantify the non-stationary (wandering) dynamics transversal to the limit cycle $\cA$ in the state plane. 
{The averaging kernel Hilbert space $\mathcal{S}_{\z,\theta}$ is a reproducing kernel Hilbert space (see \cite{Mezic_JNS2019}, Section 7), so that \eqref{eqn:KEY} holds.}
The unilateral Laplace transform $Y({s}; \x_0)$ of the output $\{y(t)=f((\x(t))\}_{t\geq 0}$ with initial condition $\x_0\in \cB(\cA)\setminus\cA$ is derived from \eqref{eqn:KEY}, \eqref{eqn:projection3}, and \eqref{eqn:expansion3} as
\begin{equation}
Y({s}; \x_0)
= [{R} ({s}; {L}_\cB)f](\x_0)
= \sum_{{m}\in\bZ}\frac{\phi_{0,{m}}(\x_0)V_{0,{m}}}{{s}-\ii{m}{\it\Omega}}
+ \sum_{\substack{k\in\bN\\ {m}\in\bZ}}
\frac{\phi_{k,{m}}(\x_0)V_{k,{m}}}{{s}-(k\nu+\ii{m}{\it\Omega})}. 
\label{eqn:expansion3-2}
\end{equation}
The first term on the right-hand side possesses pure imaginary poles only\footnote{The existence of zero pole implies non-zero time-average (mean) of $f(\x(t))$ on the limit cycle due to the underlying dynamics $\x(t)$ and the choice of observable $f$.
} and hence represents the \emph{stationary} modes evolving along the limit cycle and sustaining in time. 
The term is related to the Fourier expansion of the signal with respect to fundamental angular frequency $\it\Omega$. 
The second term includes the exponent $\nu$ in poles and hence represents the \emph{non-stationary} modes evolving to the limit cycle and decaying in time. 
As pointed out for \eqref{eqn:expansion2-2}, the Laplace transform $Y({s}; \x_0)$ in \eqref{eqn:expansion3-2} is globally valid inside the basin $\cB(\cA)$ of the nonlinear system \eqref{eqn:syst}.

%%%
\subsection{Off-Attractor Evolution to Stable Quasi-Periodic Cycle}

The above expansion is extended to the case where the system \eqref{eqn:syst} possesses a stable quasi-periodic cycle $\cA$ with a basin of attraction $\cB(\cA)$.  
The set $\cA$ corresponds to an $\ell$-dimensional torus on which the dynamics are conjugate to
\[
\dot{\theta}_j=\mathit{\Omega}_j,\quad\theta_j\in\mathbb{T},\quad\mathit{\Omega}_j\in\bR
\]
where $j=1,\ldots,\ell$, and the angular frequencies $\mathit{\Omega}_1,\ldots,\mathit{\Omega}_\ell$ satisfy the non-resonance condition. 
By the same manner as above, it is assumed that $\F$ is analytic ($\F\in\sC^\omega(\mathbb{R}^n)$) and $\sF$ is a space of observables that are $\sL^2$ in $\cA$ and $\sC^\omega$ in $\cB(\cA)\setminus{\cA}$.
Also, let $n-\ell$ quasi-characteristic exponents $\nu_1,\ldots,\nu_{n-\ell}$ of \eqref{eqn:syst} be distinct and not equal to one.  
Then, it is shown in {Theorem~9.1 of} \cite{Mezic_JNS2019} that the Koopman semigroup $\{{U}^t_\cB\}_{t\geq 0}$ (with its generator ${L}_\cB$) has a point spectrum only consisting of eigenvalues $\{(k_1\nu_1+\cdots+k_{n-\ell}\nu_{n-\ell})+\ii(m_1{\it\Omega}_1+\cdots+m_\ell{\it\Omega}_\ell\}_{k_j\in\bN_0,m_j\in\bZ}$. 
Note that we assume that the following KAM condition on the $n-\ell$ frequencies is satisfied (see {equation~(179) of} \cite{Mezic_JNS2019}): 
\begin{equation}
k_1\nu_1+\cdots+k_{n-\ell}\nu_{n-\ell} \geq c/\left(\sqrt{k^2_1+\cdots+k^2_{n-\ell}}\right)^\gamma,
\label{eqn:KMA}
\end{equation}
for some $c,\gamma>0$.
The associated Koopman eigenfunctions $\phi_{k_1\cdots k_{n-\ell},m_1\cdots m_\ell}$ satisfy the same equality as \eqref{eqn:projection3}. 
Therefore, the following theorem of the expansion of the resolvent ${R}({s}; {L}_\cB)$ holds (its proof is almost the same as for Section~\ref{thm:Stable-LC} and this omitted here).
\begin{theorem}
Consider the system \eqref{eqn:syst} with a stable quasi-periodic cycle $\cA$ and dynamics converging to it. Suppose that the KMA condition \eqref{eqn:KMA} on the $n-\ell$ frequencies of $\cA$ holds. 
Then, for the associated Koopman generator ${L}_\cB: \sD\to \mathcal{S}$, where $\mathcal{S}$ is the space of functions that are analytic (entire) in $\mathbb{C}^{n-\ell}$ and $\sL^2$ in $\mathbb{T}^\ell$, the following spectral expansion of the Koopman resolvent ${R} ({s}; {L}_\cB)$ holds:
\begin{align}
{R} ({s}; {L}_\cB) =& \sum_{\substack{k_1,\ldots,k_{n-\ell}\in\bN_0\\ m_1,\ldots,m_\ell\in\bZ}} 
{P}_{k_1\cdots k_{n-\ell},m_1,\cdots m_\ell}\times \nonumber\\
&
\times\frac{1}
{{s}-\{(k_1\nu_1+\cdots+k_{n-\ell}\nu_{n-\ell})+\ii(m_1{\it\Omega}_1+\cdots+m_\ell{\it\Omega}_\ell)\}}. 
\label{eqn:expansion4}
\end{align}
Also, the ROC of ${R} ({s}; {L}_\cB)$ {is} %corresponds to 
$\{{s}\in\bC : \mathrm{Re}[{s}]>0\}$. 
\end{theorem}

%%%
\begin{remark}
The case of a general, aperiodic---potentially chaotic---attractor can be\\ 
treated through the use of formula 
(112) in \cite{Mezic_JNS2019}, in which case the continuous spectrum can contribute to resolvent analysis.
\end{remark}

%%%%%
\section{On Computational Aspects}
\label{sec:comp}

%%%
\subsection{General Remarks}

Here, we discuss the computational aspect of the derived expansion formulae and Laplace transforms of the outputs. 
In the linear case, as shown in Section~\ref{sssec:Linear}, it is analytically possible to obtain the Laplace transforms of state and output directly from the state-space representation $(\A,{\bf c})$ in Table\,\ref{tab:idea}. 
But, in the nonlinear case, it is generally impossible to obtain analytical forms of such Laplace transforms from the state-space representation \eqref{eqn:syst}. 
Thus, it is necessary to rely on numerical computation in order to approximately evaluate their properties, which are poles (corresponding to ``Koopman eigenvalues") and residues (related to ``Koopman eigenfunctions and modes") of $Y\sub{p}({s})$ in \eqref{eqn:expansion1-2}, \eqref{eqn:expansion2-2}, and \eqref{eqn:expansion3-2}. 
The discussion can be expanded into vector-form without loss of generality; therein, the Koopman modes can have vector-form if the observable is a vector-valued function: see \eqref{eqn:expansion-vdP}.  
Also, similar requirements are associated with the function $Y\sub{c}({s})$ corresponding to the continuous spectrum in \eqref{eqn:expansion1-2}. 
In this case, the notion of poles is generalized and the eigenvalues are replaced by the continuous part of the spectrum. 

On the data-analysis side, the Koopman-operator framework has recently gained popularity in conjunction with numerical methods. 
The methods are, among others, variants of the Dynamic Mode Decomposition (DMD) \cite{Schmid_JFM656,Rowley_JFM641,Chen_JNLS22,Jovanovic_PF26,Tu_JCD1,Matt_JNLS25,%Susuki_CDC15,
Kutz:2016} and Generalized Laplace Analysis (GLA) \cite{Marko_CHAOS22,Mezic_ARFM45}.  
DMD is capable of evaluating the poles of the Laplace transforms \eqref{eqn:expansion1-2}, \eqref{eqn:expansion2-2}, and \eqref{eqn:expansion3-2}, while DMD and GLA are capable of evaluating their residues. 
The methods operate directly on data and can approximate part of the spectrum of the underlying Koopman operator.  
Convergence of the numerical methods is guaranteed in \cite{Arbabi_SIAMJADS16,Korda_JNS28} under certain conditions.

The identification of poles and residues is of great interest in science and technology: see, e.g., \cite{VanBlaricum_IEEETAP23,VanBlaricum_IEEETEMC20}. 
Since these are related to Koopman eigenvalues, eigenfunctions, and modes, they can be computed using the DMD algorithm. 
The Prony analysis \cite{Hildebrand:1956} is also formulated in \cite{Susuki_CDC15} for computing the Koopman eigenvalues and modes and used in signal processing and power grids, among others. 
For wider applications, it is required to estimate not only modes for pure-imaginary poles but also modes for poles with negative real parts. 
This is challenging from the point of view of numerics (especially in the Koopman operator framework) due to the ``non-stationary" nature of the modes. 
For instance, it is difficult to extract the poles of \eqref{eqn:expansion3-2} directly from data without knowledge on the system. 
In \cite{Korda_JNS28}, the DMD eigenvalues are proved to converge in some sense even if they have negative real-part, but these are still difficult to compute in practice.
For instance, it is an issue to clearly extract the first and second terms on the right-hand side of \eqref{eqn:expansion3-2} in terms of the time domain. 
This issue is discussed in the rest of this paper with illustrative numerical examples.

%%%
\subsection{Example of Stable Limit Cycle}
\label{sec:vdP}

Below, we discuss the result in \eqref{eqn:expansion3-2} for the van der Pol oscillator, which possesses a stable limit cycle in the state plane. 
The goal here is to compute the poles and residues, and then to decompose the dynamics into stationary (periodic) and non-stationary (transient) components. 
We first locate poles and then decompose the dynamics.   
The poles can be computed with the DMD algorithm. However, for better accuracy, we use here direct analysis of the system model, such as locating eigenvalues of adjoint (variational) equations. 
The decomposition, and in particular its non-stationary component, is obtained through the notion of \emph{isochron}, which is crucial to the analysis of asymptotically periodic systems. 
An isochron is a set of states partitioning the basin of attraction in such a way that trajectories starting from the same isochron asymptotically converge to the same orbit on the limit cycle \cite{Mauroy_CHAOS22}. 
Therein, it was shown that isochrons are level sets of a Koopman eigenfunction. 

The dynamics of the classical van der Pol oscillator is represented by
\begin{equation}
\left\{
\begin{aligned}
\dot{x}_1&=x_2,\quad \\
\dot{x}_2&={\varepsilon}(1-x^2_1)x_2-x_1,
\end{aligned}
\right.
\label{eqn:vdP}
\end{equation}
with parameter ${\varepsilon}=0.3$. 
The angular frequency of the limit cycle is ${\it\Omega}\approx 0.995$ and the negative Floquet exponent is $\nu\approx -0.301$, which are computed with FFT and variational equations, respectively. 
Here, we consider the two observables $f_1(\x)=x_1$ and $f_2(\x)=x_2$ where $\x=(x_1,x_2)^\top$, that is, the state itself. 
Then, the Laplace transform \eqref{eqn:expansion3-2} is extended in vector form as
\begin{align}
\left[\begin{array}{c}
X_1({s}; \x_0) \\
X_2({s}; \x_0)
\end{array}\right]
:=&
\left[\begin{array}{c}
[{R} ({s}; U_\cB)f_1](\x_0) \\
{[{R} ({s}; U_\cB)f_2] (\x_0)}
\end{array}\right]
\nonumber \\
=& \sum_{{m}\in\bZ}\frac{\phi_{0,{m}}(\x_0)\V_{0,{m}}}{{s}-\ii{m}{\it\Omega}}
+ \sum_{\substack{k\in\bN\\ {m}\in\bZ}}
\frac{\phi_{k,{m}}(\x_0)\V_{k,{m}}}{{s}-(k\nu+\ii{m}{\it\Omega})}. 
\label{eqn:expansion-vdP}
\end{align}
The vectors $\V_{k,m}\in\bC^2$ correspond to the Koopman modes derived by projecting the two observables $f_1$ and $f_2$ onto the eigenfunction $\phi_{k,m}$.

\begin{figure}[t]
\centering
\includegraphics[width=0.6\textwidth]{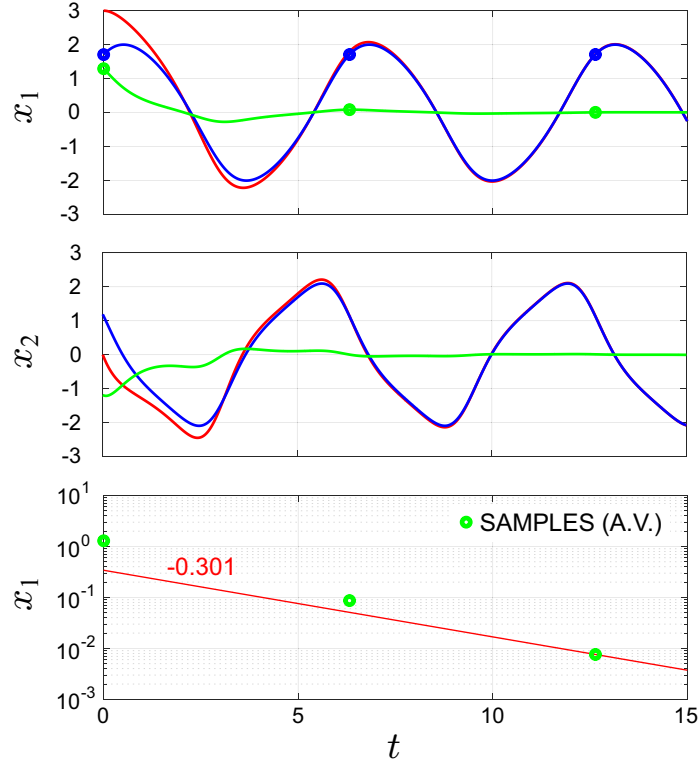}
\caption{%
Time evolutions associated with the Laplace transform of the states of van der Pol oscillator.
In the first and second figures from top, the \emph{red} lines represent a trajectory of the van der Pol oscillator. 
This time evolution can be decomposed in two parts: a stationary part (in blue) associated with Koopman eigenvalues $\ii m{\it\Omega}$ in \eqref{eqn:expansion-vdP} and a non-stationary part (in green) associated with Koopman eigenvalues $k\nu+\ii m{\it\Omega}$. 
In the third figure, the samples of the green line in $x_1$ under the period of the limit cycle are plotted in logarithmic scale, where we use absolute values. 
}%
\label{fig:vdP1}
\end{figure}

Here, we conduct a decomposition of time-domain simulation into the stationary and non-stationary parts. 
The red lines in the first and second figures of Figure\,\ref{fig:vdP1} are the time evolution $\x(t)$ of the states of the van der Pol oscillator starting from $\x(0)=(3,0)^\top$. 
The trajectory is simulated with direct numerical integration of \eqref{eqn:vdP} and converges to the limit cycle. 
The time-evolution related to the stationary part (blue lines in Figure\,\ref{fig:vdP1}) has been computed by numerical integration from an initial condition on the limit cycle and on the same isochron as the initial condition $\x(0)=(3,0)^\top$. 
The phase of the isochron can be computed through Fourier averages as in \cite{Mauroy_CHAOS22}. 
Note that it is related to the value of the associated Koopman eigenfunction, and is therefore the argument $\theta$ of the residue $\ee^{\ii\theta}$ for the pole at {$\ii m{\it\Omega}$}.  
Finally, the time evolution associated with the non-stationary part (green lines in Figure\,\ref{fig:vdP1}) is simply obtained as the difference between the time evolution of the dynamics and the stationary evolution. 
Regarding the non-stationary part, in the bottom of Figure\,\ref{fig:vdP1}, the samples of the green line in $x_1$ under the period $\it\Omega$ of the limit cycle are plotted in logarithmic scale, where we use absolute values.
We verify that it decays asymptotically, with its rate asymptotically converging to the negative Floquet exponent $\nu\approx -0.301$.  

%%%
\subsection{Example of Stable Quasi-Periodic Cycle}
\label{sec:high-d-syst}

We now consider the four-dimensional system consisting of two coupled van der Pol oscillators
\begin{equation}
\label{vdp_coupled}
\left\{
\begin{array}{rcl}
\dot{x}_1 & = & x_2\,,\\
\dot{x}_2 & = & {\varepsilon} (1-x_1^2)x_2-k_x\,x_1+k\sub{c}(y_1-x_1)\,,\\
\dot{y}_1 & = & y_2\,,\\
\dot{y}_2 & = & {\varepsilon} (1-y_1^2)y_2-k_y\,y_1+k\sub{c}(x_1-y_1)\,.
\end{array}
\right.
\end{equation}
With the parameters ${\varepsilon}=0.3$, $k_x=1$, $k_y=3$, and $k\sub{c}=0.5$, the trajectories converge to a quasiperiodic torus. 
We integrate numerically a trajectory associated with the initial condition $(3,3,3,3)$ (red curve in Figure \ref{fig:torus}). 
This initial condition is characterized by a pair of phases $(\theta_1,\theta_2)$, which are computed with the Fourier averages methods described in \cite{Mauroy_CHAOS22}. 
The stationary part associated with the trajectory corresponds to a trajectory starting on the torus with the same initial phases $(\theta_1,\theta_2)$. 
This second trajectory is also numerically integrated (blue curve in Figure \ref{fig:torus}). 
The transient part is given by the difference between these two trajectories (green curve in Figure \ref{fig:torus}(a)). 
We verify that it converges to zero, which implies that the two trajectories synchronize on the torus (Figure \ref{fig:torus}(b)).

\begin{figure}[t]
\centering
\subfigure[]{\includegraphics[width=0.48\textwidth]{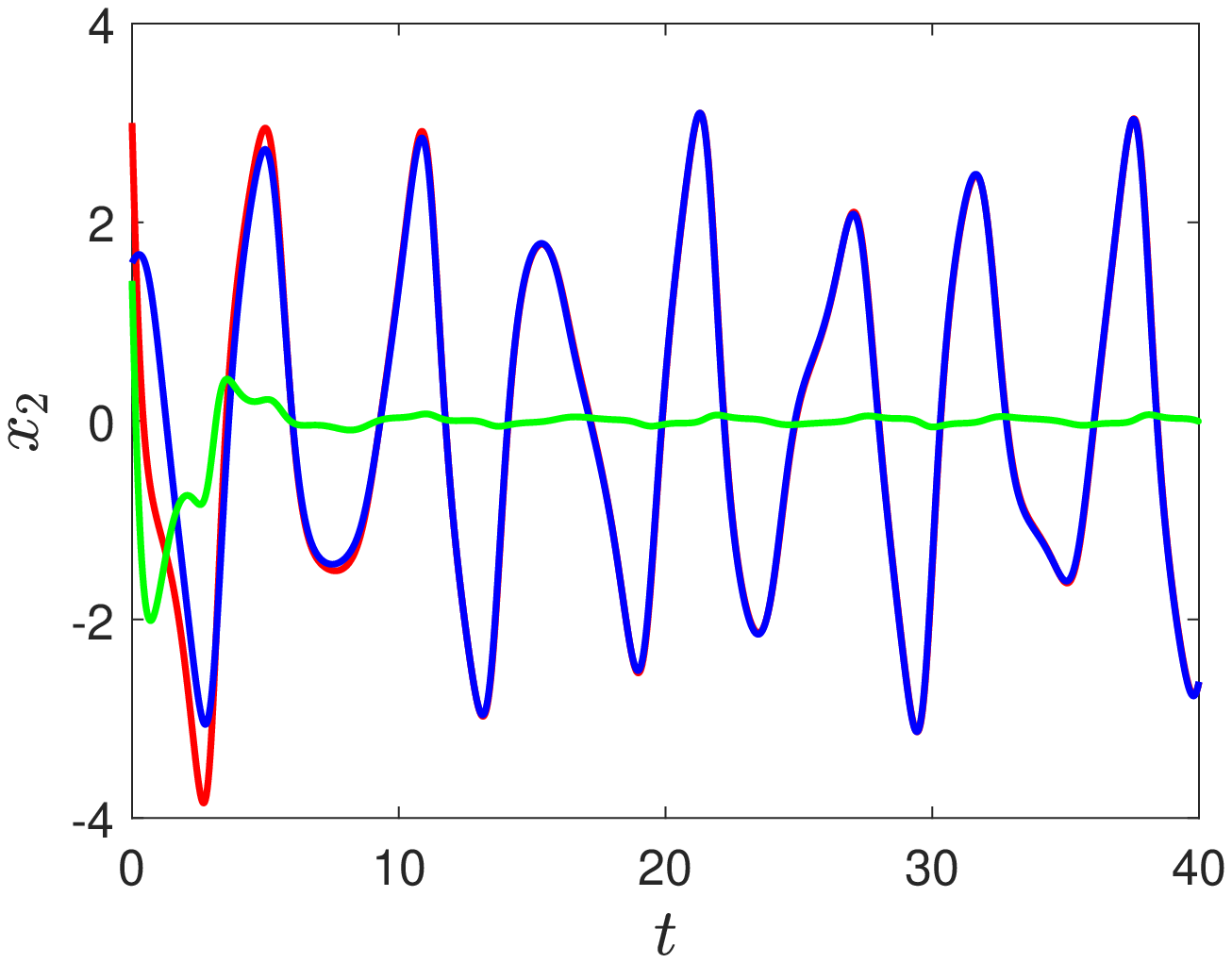}}
\subfigure[]{\includegraphics[width=0.48\textwidth]{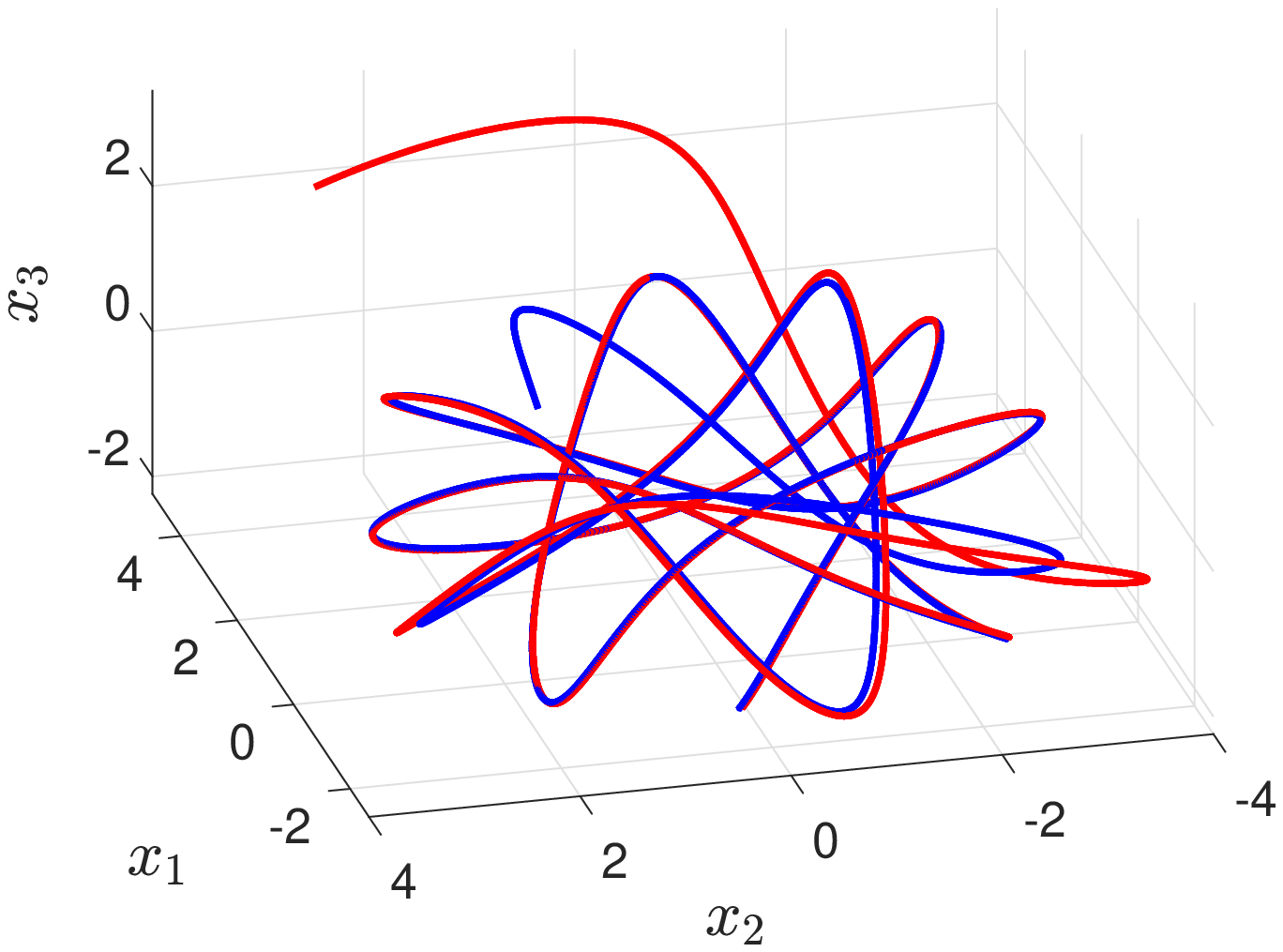}}
\caption{%
Time evolutions associated with the Laplace transform of states for two coupled van der Pol oscillators.
(a) The \emph{red} curve represents the original time evolution of the state $x_2$, the \emph{blue} curve corresponds to the stationary part of the time evolution, and the \emph{green} curve corresponds to the non-stationary part.
(b) The two trajectories (associated with identical phases) synchronize on the torus.
}%
\label{fig:torus}
\end{figure}

%%%%%
\section{Conclusion}
\label{sec:outro}

In this paper, we presented the expansion formulae of the resolvent of the Koopman operator, called the Koopman resolvent, for three types of nonlinear dynamics. 
The formulae do not rely on any linearization in the state space and thus provide the Laplace transform of the output that is clearly connected to the system properties through the Koopman operator. 
The Koopman resolvent and its connection to the nonlinear dynamics have been introduced in this paper.
The proofs provided in this paper are built on \cite{Mezic_ARFM45,Mauroy_PD261,Mezic_JNS2019} and present convergence results of the spectral expansions of the Koopman resolvent.
It is noticeable that the residue theorem was shown to work for the characterization of nonlinear dynamics even in the case when eigenvalues are not isolated (e.g., quasi-periodic case). 
Also, we discussed the computational aspect of the derived Laplace transform using numerical simulations of two simple systems with stable (quasi-)limit cycles. 
The clear separation of stationary and non-stationary parts of the Koopman modes has been emphasized. 
Therefore, we suggest that the notion of isochron is useful for the computation of the Laplace transform for nonlinear non-stationary dynamics. 

We envision several research directions.
A generalization to the case of multivariable output is trivial as shown in Section~\ref{sec:vdP}. 
The case of discrete-time systems can be formulated in the same manner as this paper using the $z$-transformation. 
Computation of the non-stationary part of off-attractor evolution to a chaotic attractor is challenging.  The extension of the Laplace-domain representation to nonlinear systems with inputs is crucial and listed up as the next topic. 
Its application to engineering of power grid dynamic performance is also a possible research perspective.
Finally, current limitations related to unstable equilibria should be considered carefully in the context of nonlinear stabilization problems (e.g. feedback stabilization).

%%%%%
{
\section*{Acknowledgments}

The authors appreciate the reviewers for their valuable suggestion of the manuscript.

%%%%%
\appendix
\section{Derivation of Spectral Expansion \eqref{eqn:on-attractor}}
\label{sec:AppA}

Formula \eqref{eqn:on-attractor} is based on the well-known Stone's theorem for any strongly continuous unitary group $\{U^t_\mathcal{A}\}_{t\in\mathbb{R}}$ in a separable Hilbert space $\mathscr{F}$ (see Theorem~1 in Section~XI-13 of \cite{Yosida:1980}). 
The theorem holds for the case of semigroup ($t\geq 0$) as stated in page~199 of \cite{Plesner:1969}:
\begin{equation}
U^t_\mathcal{A} = \ee^{\ii t H}= \int^\infty_{-\infty}\ee^{\ii \omega t}\dd E(\omega) \qquad \forall t\geq 0
\label{eqn:Stone}
\end{equation}
where $H$ is a unbounded, self-adjoint operator in $\mathscr{F}$, and $E(\omega)$ is the (canonical) resolution of the identity, precisely,  $E(\omega)=E((-\infty,\omega])$, the spectral measure of $H$ on the field $\mathscr{B}(\mathbb{R})$ of all Borel sets of $\mathbb{R}$. 
In this, $L_\mathcal{A}:=\ii H: \mathscr{D}\to\mathscr{F}$ is a unbounded, skew-adjoint operator with domain $\mathscr{D}\subset\mathscr{F}$ and 
\[
L_\mathcal{A}f :=\lim_{t\downarrow 0}\frac{U^t_\mathcal{A}f-f}{t} \qquad \forall f\in\mathscr{D}.
\]
The corollary in page~310 of \cite{Yosida:1980} states that the integral in \eqref{eqn:Stone} implies convergence in strong operator topology, i.e., for $\forall f\in\mathscr{F}$, 
\[
\lim_{R\to\infty}\left\|\int^{R}_{-R}\ee^{\ii \omega t}\dd F(\omega)f - U^t_\mathcal{A}f\right\|=0 \qquad \forall t\geq 0.
\]
From pages~140--141 of \cite{Plesner:1969}, the self-adjoint operator $H$ has the decomposition as
\[
H = H\sub{p}+H\sub{c},
\]
with
\begin{equation}
H\sub{p} :=\sum_{\omega\in n(H)\subset\sigma\sub{p}(H)} \omega P_\omega \label{eqn:Ap}
\end{equation}
where $n(H)$ is the set of all eigenvalues of $H$ with $\mathrm{cl}(n(H))=\sigma\sub{p}(H)$ (point spectrum of $H$), and 
$\displaystyle P_\omega:=E(\omega)-E(\omega-0)$ where $\displaystyle \lim_{\lambda\uparrow\omega}\|[E(\lambda)-E(\omega-0)]f\|=0$ for all $f\in\mathscr{F}$ is the projection operator induced on the eigen-space $\mathscr{F}_\omega$ corresponding to an eigenvalue $\omega\in\mathbb{R}$.  
Here we have used the assumption of ergodicity of the invariant measure, implying that every eigenvalue of $H$ is simple (see Theorem~9.21 of \cite{Arnold:1968}).  
Also, from the self-adjoint property of $H$ on $\sF=\sL^2$, %From this, 
we can use the label $j$ of natural numbers to represent the eigenvalues as $\omega_j${, and the} %. 
%The 
$P_{\omega_j}$ are mutually orthogonal. %from the self-adjoint property of $H$. 
That is, $H\sub{p}$ is associated with the point spectrum of $H$, induced on the Hilbert space
\[
\mathscr{F}\sub{p}:=\bigoplus_{j=1}^\infty\mathscr{F}_{\omega_j},
\]
while $H\sub{c}$ is with the continuous spectrum of $H$ and induced on the Hilbert space $\mathscr{F}\sub{c}:=\mathscr{F}\setminus\mathscr{F}\sub{p}$, which is the orthogonal complement of $\mathscr{F}\sub{p}$. 
Note that $H\sub{p}$, namely, the series on the right-hand side of \eqref{eqn:Ap} exists in the sense of strong operator topology because the set $\{\phi_{\omega_j}\in\mathscr{F}\setminus\{0\}\}_{j=1,2,\ldots}$ of eigenfunction $\phi_{\omega_j}$ of the self-adjoint operator $H$ corresponding to an eigenvalue $\omega_j$, defining as $H\phi_{\omega_j} = \omega_j\phi_{\omega_j}$, is an orthonormal basis of $\mathscr{F}\sub{p}$. 
Now we set
\[
E\sub{p}({\it\Lambda}):=\sum_{\omega_j\in(n(A)\cap{\it\Lambda})}P_{\omega_j}, \qquad
E\sub{c}({\it\Lambda}):=E({\it\Lambda})-E\sub{p}({\it\Lambda})
\]
where ${\it\Lambda}\in\mathscr{B}(\mathbb{R})$, and $E\sub{p}({\it\Lambda})$ and $E\sub{c}({\it\Lambda})$ are orthogonal measures on $\mathscr{F}$; in particular, $E\sub{p}(\mathbb{R})$ and $E\sub{c}(\mathbb{R})$ are the projection operators on $\mathscr{F}\sub{p}$ and $\mathscr{F}\sub{c}$. 
The measure $E\sub{p}({\it\Lambda})$ is a point measure and might be considered as acting on $\mathscr{B}(\mathbb{R})$, while the domain of $E\sub{c}({\it\Lambda})$ is $\mathscr{B}(\mathbb{R})$. 
Since $H\sub{c}$ has no point spectrum by construction, the measure $E\sub{c}(\omega)=E\sub{c}((-\infty,\omega])$ is continuous, that is, $E\sub{c}(\omega)=E\sub{c}(\omega-0)$ for all $\omega\in\mathbb{R}$.  
By substituting the decomposition $E({\it\Lambda})=E\sub{p}({\it\Lambda})+E\sub{c}({\it\Lambda})$ into the integral representation \eqref{eqn:Stone} we obtain \eqref{eqn:on-attractor}:
\[
U^t_\mathcal{A}
=\sum_{j=1}^\infty\ee^{\ii \omega_jt} P_{\omega_j}
+\int^\infty_{-\infty}\ee^{\ii \omega t}\dd E\sub{c}(\omega)
\qquad \forall t\geq 0.
\]
Note that in \eqref{eqn:on-attractor} we use $P_j$ and $E$ instead of $P_{\omega_j}$ and $E\sub{c}$ for simple notation.
The derivation shows that because every eigenvalue is assumed to be simple, \eqref{eqn:on-attractor} is specific to the Koopman semigroup for ergodic dynamics and does not hold for any strongly continuous unitary semigroup in a Hilbert space. 
}

%%%%%
%\bibliographystyle{siamplain}
%\bibliography{../../text/mybib/mezic,../../text/mybib/main-SSK,../../text/mybib/koopman,../../text/mybib/power}
%\input{smm_ver0506.bbl}

%%%%%
\bibliographystyle{unsrt}
\bibliography{../../../text/mybib/mezic,../../../text/mybib/main-SSK,../../../text/mybib/koopman,../../../text/mybib/power}

\end{document}